\documentclass[]{elsarticle}
\usepackage[utf8]{inputenc}
\usepackage[T1]{fontenc}
\usepackage{amsmath}
\usepackage{amssymb}
\usepackage{mathtools}
\usepackage[ruled,lined,linesnumbered]{algorithm2e}
\usepackage[colorlinks=true, linkcolor=red, citecolor=blue, urlcolor=cyan]{hyperref}
\usepackage{amsthm}
\usepackage{thmtools}
\usepackage{geometry}
\usepackage{cleveref} % better referencing

\makeatletter
\g@addto@macro\bfseries{\boldmath}
\makeatother

\DeclarePairedDelimiter{\abs}{\lvert}{\rvert}

\declaretheorem[name=Theorem,numberwithin=section]{theorem}
\declaretheorem[name=Proposition,sibling=theorem,refname={proposition,propositions}]{proposition}
\declaretheorem[name=Lemma,sibling=theorem]{lemma}
\declaretheorem[name=Corollary,sibling=theorem]{cor}
\declaretheorem[name=Definition,sibling=theorem,style=definition]{definition}
\declaretheorem[name=Remark,sibling=theorem,style=definition]{remark}

\declaretheorem[name=Example,sibling=theorem,style=definition]{example}

\title{Constructive Arithmetics in Ore Localizations of Domains}
\author[RWTH]{Johannes Hoffmann}
\ead{Johannes.Hoffmann@math.rwth-aachen.de}
\author[RWTH]{Viktor Levandovskyy}
\ead{Viktor.Levandovskyy@math.rwth-aachen.de}
\address[RWTH]{Lehrstuhl D f\"ur Mathematik, RWTH Aachen University}

\begin{document}

\begin{abstract}
For a non-commutative domain $R$ and a multiplicatively closed set $S$ the (left) Ore localization of $R$ at $S$ exists if and only if $S$ satisfies the (left) Ore property.
Since the concept has been introduced by Ore back in the 1930's, Ore localizations have been widely used in theory and in applications.
We investigate the arithmetics of the localized ring $S^{-1}R$ from both theoretical and practical points of view.
We show that the key component of the arithmetics is the computation of the intersection of a left ideal with a submonoid $S$ of $R$.
It is not known yet, whether there exists an algorithmic solution of this problem in general.
Still, we provide such solutions for cases where $S$ is equipped with additional structure by distilling three most frequently occurring types of Ore sets.
% and provide partial solutions to the mentioned problem for these types.
We introduce the notion of the (left) saturation closure and prove that it is a canonical form for (left) Ore sets in $R$. 
We provide an implementation of arithmetics over the ubiquitous $G$-algebras in \textsc{Singular:Plural} and discuss questions arising in this context.
Numerous examples illustrate the effectiveness of the proposed approach. 
\end{abstract}

\maketitle

\setcounter{tocdepth}{2}
\tableofcontents

\section*{Introduction}

In the beginning of the 1930's {\O}ystein Ore introduced several algebraic concepts \cite{ore31, ore33}, which have seriously influenced the development of algebra and its
applications.
One of them, an Ore extension of a ring, proved to be a very useful generalization of the construction of commutative rings.
Another one is Ore localization, which is utilized very widely from ring theory to algebras of operators, arising in algebraic analysis and algebraic combinatorics.
The very formalism of Ore localization was theoretically constructive in its appearance.
While computations with finitely presented algebras form a part of computer algebra, traditionally assisted by (non-commutative) Gr\"obner bases, localization in general allows us to recognize the structure of objects in a variety of non finitely presented algebras.
The latter has been intensively used in algebraic geometry and commutative algebra, also accompanied by algorithms and implementations from the 1980's, see e.g. \cite{GPS08}. 
It is natural to apply the same philosophy to non-commutative rings, and with this paper we present our investigations for domains.
The major task, which we call our \textbf{Ore Dream}, consists in the following: provide procedures and, ideally, algorithms and computer-assisted tools for manipulating left or right fractions in an Ore localization of a domain with respect to a (multiplicatively closed) Ore set $S$.

We investigate this task in detail for a domain $R$ and identify a key problem for algorithmic computations: the intersection of a left ideal with a submonoid $S$ of $R$.
To the best of our knowledge no algorithmic solution to this problem exists if only the monoid structure of $S$ is taken into account.
We propose to specify a \emph{type} of an Ore set according to the presence of additional algebraic structure and address three common types which appear most frequently in applications.
For each of these we provide a solution to the key problem and discuss the occuring limitations.

The need for Gr\"obner bases over domains and, in particular, of elimination and syzygies inspired the restriction of the rings under consideration to the broad class of ubiquitous $G$-algebras (cf. Section \ref{sect_G-algebras}).

Historically, perhaps the first connection between the arithmetic operations in the quotient field (which is an example of Ore localization) of a Noetherian domain $R$ and syzygies over $R$ was the paper \cite{AL88} by Apel and Lassner.
They have analyzed the case where $R$ is a universal enveloping algebra of a finitely dimensional Lie algebra.
Notably, the extension of these results to the whole class of PBW algebras was completed in \cite{BGV}.

We analyze the approach to arithmetics of fractions in an Ore localization from the point of view of computability.
Moreover, we present an implementation {\tt olga.lib} in the computer algebra system {\sc Singular:Plural} \cite{Plural}.
To the best of our knowledge, apart from {\tt olga.lib} and {\sc JAS} (\cite{JAS}), which performs similar computations even over {\em parametric solvable polynomial rings} \cite{Kredel2015}, no other package can offer constructive computations on that level of generality.
However, the price we pay for this is high: generally, Gr\"obner bases over related rings are invoked for manipulations with fractions both in {\sc Plural} and in {\sc JAS}.

There are several packages for computer algebra systems dealing with similar situations, most notably {\sc OreTools} \cite{OreTools} and {\sc OreAlgebra} \cite{Mgfun} in {\sc MAPLE}, {\sc ore$\_$algebra} \cite{OreAlgebraSage} in {\sc SAGE}, and {\sc HolonomicFunctions} \cite{HoloFun} in {\sc MATHEMATICA}.
These work over predefined algebras, such as univariate algebras of operators (differential, difference and $q$-difference among most prominent ones, cf. \cite{Mgfun}) over a field of rational functions as coefficient domain (these rings are also Ore localizations).
In such situations, as investigated e.g. in \cite{Grig90, vdH16}, one can even estimate the complexity of operations.
In contrast our development serves a general purpose; in the future one could develop specialized better algorithms for new algebras and/or their Ore subsets.

This paper is an extended, expanded and enhanced version of the paper \cite{JHL17}, which appeared at ISSAC 2017 in Kaiserslautern, Germany.
Proofs have been either restored from the abridged version or expanded in details.
A new part on the simplifying procedure for fractions has been added.
We enhanced presented examples and added a new \Cref{case_study_Weyl} devoted to a lively case study.
We updated the exposition with recent results and publications in the area.
In the meantime our implementation \texttt{olga.lib} also has been significantly improved.

\section{Basics of left Ore localization}

In this section we recall the classical material based on Ore's original paper \cite{ore31} following a modern exposition inspired by \cite{BGV}:

\begin{definition}
	Let $R$ be a domain.
	A subset $S$ of $R$ is called \emph{multiplicatively closed} if $1\in S$, $0\notin S$ and for all $s,t\in S$ we have $s\cdot t\in S$.
	Furthermore, $S$ is called a \emph{left Ore set} if it is multiplicatively closed and satisfies the \emph{left Ore condition}: for all $s\in S$ and $r\in R$ there exist $\tilde{s}\in S$ and $\tilde{r}\in R$ such that $\tilde{s}r=\tilde{r}s$.
\end{definition}

Any subset $B$ of $R\setminus\{0\}$ has a minimal multiplicatively closed superset $[B]$, which consists of all finite products of elements of $B$, where the empty product represents $1$.

\begin{definition}
	Let $S$ be a multiplicatively closed subset of a domain $R$.
	A ring $R_S$ together with an injective homomorphism $\varphi:R\rightarrow R_S$ is called a \emph{left Ore localization} of $R$ at $S$ if:
	\begin{enumerate}
		\item
			For all $s\in S$, $\varphi(s)$ is a unit in $R_S$.
		\item
			For all $x\in R_S$ there exist $s\in S$ and $r\in R$ such that $x=\varphi(s)^{-1}\varphi(r)$.
	\end{enumerate}
\end{definition}

One can show that the Ore localization of $R$ at $S$ exists if and only if $S$ is a left Ore set in $R$.
In this case, the localization is unique up to isomorphism.
The classical construction is given by the following:

\begin{theorem}\label{construction_of_Ore_localization}
	Let $S$ be a left Ore set in a domain $R$.
	The relation $\sim$ on $S\times R$, given by
	\[
		(s_1,r_1)\sim(s_2,r_2)
		\Leftrightarrow
		\exists\tilde{s}\in S~\exists\tilde{r}\in R:
		\tilde{s}s_2=\tilde{r}s_1\text{ and }\tilde{s}r_2=\tilde{r}r_1,
	\]
	is an equivalence relation.
	Now $S^{-1}R:=((S\times R)/\sim,+,\cdot)$ becomes a ring via
	\[
		(s_1,r_1)+(s_2,r_2)
		:=(\tilde{s}s_1,\tilde{s}r_1+\tilde{r}r_2),
	\]
	where $\tilde{s}\in S$ and $\tilde{r}\in R$ satisfy $\tilde{s}s_1=\tilde{r}s_2$, and
	\[
		(s_1,r_1)\cdot(s_2,r_2)
		:=(\tilde{s}s_1,\tilde{r}r_2),
	\]
	where $\tilde{s}\in S$ and $\tilde{r}\in R$ satisfy $\tilde{r}s_2=\tilde{s}r_1$.
	Together with the injective \emph{structural homomorphism}
	\[
		\rho_{S,R}:R\rightarrow S^{-1}R,\quad
		r\mapsto(1,r),
	\]
	$(S^{-1}R,\rho_{S,R})$ is the left Ore localization of $R$ at $S$.
\end{theorem}

The elements of $S^{-1}R$ are called \emph{left fractions} and are denoted by $s^{-1}r$ or, by abuse of notation, again by $(s,r)$.
Some basic facts concerning the localization are collected below:

\begin{lemma}
	Let $S$ be a left Ore set in a domain $R$ and $(s,r)\in S^{-1}R$.
	\begin{enumerate}[(a)]
		\item %[(a)]
			$0_{S^{-1}R}=(1_R,0_R)$ and $1_{S^{-1}R}=(1_R,1_R)$.
		\item %[(b)]
			$(s,r)=1$ if and only if $s=r$.
		\item %[(c)]
			$(s,r)=0$ if and only if $r=0$.
		\item %[(d)]
			Let $t\in R$.
			If $ts\in S$, then $(s,r)=(ts,tr)$.
		\item %[(e)]
			$-(s,r)=(s,-r)$.
		\item %[(f)]
			$S^{-1}R$ is a domain.
	\end{enumerate}
\end{lemma}

According to the previous lemma, additive inverses of left fractions are quite easy to find, but what about multiplicative inverses?
If $r\in S$, then the inverse of $(s,r)$ is given by $(r,s)$.
But there might be other invertible left fractions whose numerators do not belong to $S$, a phenomenon that even occurs in commutative localizations:

\begin{example}
	Consider the localization $K[x]_{x^2}:=[x^2]^{-1}K[x]$, then $x\notin[x^2]$, but $(1,x)$ is invertible with inverse $(x^2,x)$.
\end{example}

We turn to the theory of left saturation closure to find a complete description of the invertible elements of the localization.

\section{A brief introduction to the left saturation closure}

From this point on we present new results unless stated otherwise.

In this section let $R$ be a domain.

\begin{definition}
	A subset $S$ of $R$ with $0\notin S$ is called \emph{left (resp. right) saturated}, if for all $a,b\in R$, $a\cdot b\in S$ implies $b\in S$ (resp. $a\in S$).
	Furthermore, $S$ is called \emph{saturated} if it is both left and right saturated.
\end{definition}

While the notion of multiplicative closure only depends on the multiplication and is unchanged under embedding $R$ into a larger ring, being saturated involves factorization and thus heavily depends on the context: the set $S:=\mathbb{Z}\setminus\{0\}$ is both multiplicatively closed and saturated in $\mathbb{Z}$.
In $\mathbb{Q}$ it is still multiplicatively closed, but no longer saturated, since $2\cdot\frac{1}{2}=1\in S$, but $\frac{1}{2}\notin S$.

\begin{definition}
	Let $S$ be a multiplicatively closed subset of $R$.
	The \emph{left saturation closure} of $S$ is
	\[
		\operatorname{LSat}(S)
		:=\{r\in R\mid\exists w\in R:wr\in S\}.
	\]
\end{definition}

Since $1\in S\subseteq R$ we have $S\subseteq\operatorname{LSat}(S)$, in particular $1\in\operatorname{LSat}(S)$.
Furthermore, $0\notin\operatorname{LSat}(S)$ since $0\notin S$.

The following lemma justifies the name ``left saturation closure''.

%\begin{lemma}\label{LSat_is_left_saturation_closure}
%	Let $S$ be a multiplicatively closed subset of $R$.
%	\begin{enumerate}
%		\item[(a)]
%			$\operatorname{LSat}(S)$ is the smallest left saturated superset of $S$ with respect to inclusion.
%		\item[(b)]
%			$S$ is left saturated if and only if $S=\operatorname{LSat}(S)$.
%	\end{enumerate}
%\end{lemma}

\begin{lemma}\label{LSat_is_left_saturation_closure}
	Let $S$ be a multiplicatively closed subset of $R$.
	\begin{enumerate}[(a)]
		\item
			$\operatorname{LSat}(S)$ is left saturated.
		\item
			$S$ is left saturated if and only if $S=\operatorname{LSat}(S)$.
		\item
			$\operatorname{LSat}(S)$ is the smallest left saturated superset of $S$ with respect to inclusion.
	\end{enumerate}
\end{lemma}
\begin{proof}
	\begin{enumerate}[(a)]
		\item
			Let $a,b\in R$ such that $ab\in\operatorname{LSat}(S)$, then there exists $w\in R$ such that $wab\in S$, thus $b\in\operatorname{LSat}(S)$.
		\item
			Let $r\in\operatorname{LSat}(S)$, then there exists $w\in R$ such that $wr\in S$.
			If $S$ is left saturated this implies $r\in S$, thus $S=\operatorname{LSat}(S)$.\\
			For the reverse implication note that $\operatorname{LSat}(S)$ is left saturated by the previous result.
			If $S=\operatorname{LSat}(S)$, then so is $S$.
		\item
			Let $Q\subseteq R$ be a left saturated set such that $S\subseteq Q\subseteq\operatorname{LSat}(S)$.
			Let $r\in\operatorname{LSat}(S)$, then there exists $w\in R$ such that $wr\in S\subseteq Q$.
			Since $Q$ is left saturated we have $r\in Q$ and thus $Q=\operatorname{LSat}(S)$.
	\end{enumerate}
\end{proof}

Denote the set of units of $R$ by $U(R)$.
Now we answer the question posed at the end of the last section concerning invertible elements:

\begin{proposition}\label{fraction_invertible_iff_numerator_in_LSat}
	Let $S$ be a left Ore set in $R$ and $(s,r)\in S^{-1}R$.
	Then the following are equivalent:
	\begin{enumerate}[(1)]
		\item
			$(s,r)\in U(S^{-1}R)$.
		\item
			$(1,r)\in U(S^{-1}R)$.
		\item
			$r\in\rho_{S,R}^{-1}(U(S^{-1}R))\quad\Leftrightarrow\quad(1,r)=\rho_{S,R}(r)\in U(S^{-1}R)$.
		\item
			$r\in\operatorname{LSat}(S)$.
	\end{enumerate}
\end{proposition}
\begin{proof}
	Statements (2) and (3) are equivalent since $\rho_{S,R}(r)=(1,r)$.
	Furthermore, $(s,r)=(s,1)\cdot(1,r)$, where $(s,1)\in U(S^{-1}R)$ with inverse $(1,s)$, which shows the equivalence of (1) and (2).
	
	Starting from (2), let $(1,r)\in U(S^{-1}R)$, then there exists $(s,w)\in S^{-1}R$ such that $1=(s,w)\cdot(1,r)=(s,wr)$, which implies $wr=s\in S$, thus $r\in\operatorname{LSat}(S)$ and we have reached (4).
	
	For the reverse implication, let $r\in\operatorname{LSat}(S)$, then there exists $w\in R$ such that $wr\in S$.
	Now $(wr,w)\cdot(1,r)=(wr,wr)=1$ implies that $(1,r)\in U(S^{-1}R)$.
\end{proof}

This implies that the left saturation closure of left Ore sets is actually saturated on both sides:

\begin{lemma}
	Let $S$ be a left Ore set in $R$.
	Then $\operatorname{LSat}(S)$ is saturated.
\end{lemma}
\begin{proof}
	Let $a,b\in R$ such that $a\cdot b\in\operatorname{LSat}(S)$.
	Then $(1,ab)$ is a unit in $S^{-1}R$ by the previous result.
	Now $(1,a)\cdot(1,b)=(1,ab)$ implies that both $(1,a)$ and $(1,b)$ are also units in $S^{-1}R$ since $S^{-1}R$ is a domain, thus $a,b\in\operatorname{LSat}(S)$ again by the previous result.
\end{proof}

Therefore a left Ore set $S$ is saturated if and only if $S=\operatorname{LSat}(S)$.

Apart from this characterization, the left saturation closure has even more interesting applications.
For a start we see that $\operatorname{LSat}$ preserves and reflects the left Ore condition:

\begin{lemma}\label{S_Ore_iff_LSat(S)_Ore}
	Let $S$ be a multiplicatively closed subset of $R$.
	Then $S$ satisfies the left Ore condition in $R$ if and only if $\operatorname{LSat}(S)$ satisfies the left Ore condition in $R$.
\end{lemma}
\begin{proof}
	Let $x\in\operatorname{LSat}(S)$, $r\in R$ and $w\in R$ such that $wx\in S$.
	If $S$ satisfies the left Ore condition, then there exist $\tilde{s}\in S$ and $\tilde{r}\in R$ such that $\tilde{s}r=\tilde{r}wx$.
	Since $\tilde{s}\in S\subseteq\operatorname{LSat}(S)$ and $\tilde{r}w\in R$ this implies that $\operatorname{LSat}(S)$ satisfies the left Ore condition.
	
	For the other implication, let $r\in R$ and $s\in S\subseteq\operatorname{LSat}(S)$.
	If $\operatorname{LSat}(S)$ satisfies the left Ore condition, then there exist $x\in\operatorname{LSat}(S)$ and $\tilde{r}\in R$ such that $xr=\tilde{r}s$.
	Let $w\in R$ such that $\hat{s}:=wx\in S$ and define $\hat{r}:=w\tilde{r}\in R$, then $\hat{s}r=wxr=w\tilde{r}s=\hat{r}s$ shows that $S$ satisfies the left Ore condition.
\end{proof}

While the left saturation closure of multiplicatively closed sets is not multiplicatively closed in general, the left Ore condition is sufficient to overcome this problem.

\begin{proposition}\label{localization_at_left_saturation}
	Let $S$ be a left Ore set in $R$.
	Then $\operatorname{LSat}(S)$ is a left Ore set in $R$ and $S^{-1}R\cong\operatorname{LSat}(S)^{-1}R$.
\end{proposition}
\begin{proof}
	For the first part it remains to show that $\operatorname{LSat}(S)$ is multiplicatively closed: let $x,y\in\operatorname{LSat}(S)$, then there exist $a,b\in R$ such that $ax,by\in S$.
	By the left Ore condition on $S$ there exist $\tilde{s}\in S$ and $\tilde{r}\in R$ such that $\tilde{s}b=\tilde{r}ax$.
	Now $w:=\tilde{r}a\in R$ and $wxy=\tilde{s}by\in S$ shows that $xy\in\operatorname{LSat}(S)$.
	For the second part, consider the map
	\[
		\psi:S^{-1}R\rightarrow\operatorname{LSat}(S)^{-1}R,\quad(s,r)\mapsto(s,r),
	\]
	which can be shown to be an injective homomorphism of rings with standard Ore-style calculations.
	To see surjectivity, consider a fraction $(x,r)\in\operatorname{LSat}(S)^{-1}R$ and $w\in R$ such that $wx\in S$, then $(x,r)=\psi((wx,wr))$.
\end{proof}

From this we immediately get a sufficient condition for two localizations of $R$ being isomorphic:

\begin{cor}\label{LSat_equality_implies_localizations_isomorphic}
	Let $S,T$ be left Ore sets in a domain $R$.
	If $\operatorname{LSat}(S)=\operatorname{LSat}(T)$, then $S^{-1}R\cong T^{-1}R$.
\end{cor}

\begin{lemma}\label{S_in_LSat(T)_iff_LSat(S)_in_LSat(T)}
	Let $S,T$ be multiplicatively closed sets in $R$.
	Then $S\subseteq\operatorname{LSat}(T)$ if and only if $\operatorname{LSat}(S)\subseteq\operatorname{LSat}(T)$.
\end{lemma}
\begin{proof}
	Let $S\subseteq\operatorname{LSat}(T)$ and $x\in\operatorname{LSat}(S)$, then there exists $w\in R$ such that $wx\in S\subseteq\operatorname{LSat}(T)$.
	But then $vwx\in T$ for some $v\in R$, which implies $x\in\operatorname{LSat}(T)$ and thus $\operatorname{LSat}(S)\subseteq\operatorname{LSat}(T)$.
	The other implication is obvious from $S\subseteq\operatorname{LSat}(S)$.
\end{proof}

\begin{remark}
\label{rem:repr}
From a theoretical viewpoint the left saturation closure is a powerful tool that gives us a \emph{canonical form} of left Ore sets with respect to the corresponding localization.
For instance, in Proposition \ref{fraction_invertible_iff_numerator_in_LSat} we have seen that knowing $\operatorname{LSat}(S)$ is equivalent to knowing $U(S^{-1}R)$.
Unfortunately, the left saturation closure is - depending on the situation - difficult or maybe even impossible to compute or even represent in finite terms.
In \Cref{case_study_Weyl} we discuss a left Ore set, generated by two elements, which has a infinitely (though countably) generated left saturation closure.
In general we do not even expect the saturation closure to be countably generated. 

In our opinion, these facts need not be perceived as \emph{``bad news''}, but rather as an indication that the objects we are dealing with are intrinsically complicated.
\end{remark}

\section{A constructive approach to the left Ore condition}

Given a left Ore set $S$ in a domain $R$ and $(s,r)\in S\times R$, we are interested in constructively finding solutions of the left Ore condition.
We will consider the following sub-problems:
\begin{enumerate}[(1)]
	\item
		Find $\tilde{s}\in S$ such that there exists $\tilde{r}\in R$ satisfying $\tilde{s}r=\tilde{r}s$.
	\item
		Find the set of all $\tilde{s}\in S$ such that there exists $\tilde{r}\in R$ satisfying $\tilde{s}r=\tilde{r}s$.
	\item
		Given $\tilde{s}\in S$ such that there exists $\tilde{r}\in R$ satisfying $\tilde{s}r=\tilde{r}s$, find such $\tilde{r}$.
\end{enumerate}
From a theoretic viewpoint, all solutions of the left Ore condition are equivalent in the sense of Theorem \ref{construction_of_Ore_localization}:

\begin{lemma}
	Let $S$ be a left Ore set in a domain $R$ and $(s,r)\in S\times R$.
	Furthermore, let $(s_1,r_1),(s_2,r_2)\in S\times R$ such that $s_1r=r_1s$ and $s_2r=r_2s$, then $(s_1,r_1)\sim(s_2,r_2)$.
\end{lemma}
\begin{proof}
	Let $\hat{s}\in S$ and $\hat{r}\in R$ such that $\hat{s}s_1=\hat{r}s_2$, then
	\[
		\hat{s}r_1s
		=\hat{s}s_1r
		=\hat{r}s_2r
		=\hat{r}r_2s
	\]
	implies $\hat{s}r_1=\hat{r}r_2$ and thus $(s_1,r_1)\sim(s_2,r_2)$.
\end{proof}

Nevertheless, the algorithms presented later strive to compute a solution to problem (2) above, since even checking the equivalence of two left fractions is a non-trivial task.

As a last observation, note that the solution to problem (3) is unique: if there exist $r_1,r_2\in R$ such that $r_2s=\tilde{s}r=r_1s$, then $(r_1-r_2)s=0$, which implies $r_1=r_2$ since $R$ is a domain and $s\neq0$.

\subsection{The kernel technique}

Let $S$ be a multiplicatively closed subset of a domain $R$.
To the best of our knowledge there is no algorithm to decide whether $S$ is a left Ore set in $R$.
Usually such facts are established by the means of theoretical proofs.

To deal with this issue, consider the map
\[
	\varphi_{s,r}:R\overset{\cdot r}{\longrightarrow}R/Rs,\quad
	x\mapsto xr+Rs.
\]
It is a homomorphism of left $R$-modules and the intersection of its kernel with $S$ is exactly the solution of (2) from above.
This immediately gives us the following characterization of the left Ore property:

\begin{proposition}\label{Ore_condition_iff_all_kernels_intersect}
	The following are equivalent:
	\begin{enumerate}[(1)]
		\item
			$S$ satisfies the left Ore condition in $R$.
		\item
			For all $(s,r)\in S\times R$, $\ker(\varphi_{s,r})\cap S\neq\emptyset$.
	\end{enumerate}
\end{proposition}

\begin{remark}
	Provided we can check algorithmically whether $\ker(\varphi_{s,r})\cap S$ is empty, Proposition \ref{Ore_condition_iff_all_kernels_intersect} also allows us to constructively prove that a given set $S$ is \emph{not} a left Ore set in $R$, if we (correctly) suspect $s$ and $r$ to violate the Ore condition (see \Cref{example_geometric_shift}).
\end{remark}

After choosing a $\tilde{s}\in\ker(\varphi_{s,r})\cap S$, $\tilde{r}$ is the unique solution of the equation $\tilde{s}r=\tilde{r}s$.
These considerations are combined in the procedure \textsc{LeftOre}.

\begin{algorithm}
	\KwIn{$(s,r)\in S\times R$.}
	\KwOut{$(\tilde{s},\tilde{r})\in S\times R$ such that $\tilde{s}r=\tilde{r}s$.}
	\Begin{
		let $\varphi:R\rightarrow R/Rs,~x\mapsto xr-Rs$\;
		compute $\ker(\varphi)$\;
		compute any $\tilde{s}\in\ker(\varphi)\cap S$\;
		compute the unique solution $\tilde{r}\in R$ of the equation $\tilde{s}r=\tilde{r}s$\;
		\Return{$(\tilde{s},\tilde{r})$}
	}
	\caption{\textsc{LeftOre}}
\end{algorithm}

%\begin{remark}
%	Instead of $\varphi_{s,r}$ one could also consider the seemingly similar map
%	\[
%		\psi_{s,r}:R\overset{\cdot s}{\longrightarrow}R/Rr,\quad
%		x\mapsto xs+Rr.
%	\]
%	Here we avoid intersecting $\ker(\psi_{s,r})$ with $S$, but the kernel may contain numerous false candidates for $\tilde{r}$: although there exists $y\in R$ such that $yr=\tilde{r}s$, there is no guarantee that we can find such $y$ in $S$.
%	For this reason we prefer to work with $\varphi_{s,r}$.
%\end{remark}

\begin{remark}
	Instead of $\varphi_{s,r}$ one could also consider the seemingly similar map
	\[
		\psi_{s,r}:R\overset{\cdot s}{\longrightarrow}R/Rr,\quad
		x\mapsto xs+Rr.
	\]
	Here we avoid intersecting $\ker(\psi_{s,r})$ with $S$ at first, but the kernel may contain numerous false candidates for $\tilde{r}$: although there exists $y\in R$ such that $yr=\tilde{r}s$, there is no guarantee that we can find such $y$ in $S$.
	Therefore we need to intersect $M:=\{(\tilde{r},\tilde{s})\in\ker(\psi_{s,r})\times R\mid\tilde{r}s=\tilde{s}r\}$ with $R\times S$, which poses the same problems as intersecting $\ker(\varphi_{s,r})$ with $S$ directly.
	For this reason we prefer to work with $\varphi_{s,r}$.
\end{remark}

To be able to actually carry out the computations from the procedure \textsc{LeftOre} and turn it into an algorithm we need to work in a computation-friendly setting in which all intermediate steps can be carried out.

\subsection{The class of G-algebras}\label{sect_G-algebras}

\begin{definition}
	For a field $K$, $n\in\mathbb{N}$ and $1\leq i<j\leq n$ consider non-zero constants $c_{ij}\in K$ and polynomials $d_{ij}\in K[x_1,\ldots,x_n]$.
	Suppose that there exists a monomial total well-ordering $<$ on $K[x_1,\ldots,x_n]$, such that for any $1\leq i<j\leq n$ either $d_{ij}=0$ or the leading monomial of $d_{ij}$ with respect to $<$ is smaller than $x_i x_j$.
	The $K$-algebra
	\[
		A
		:=K\langle x_1,\ldots,x_n\mid\{x_jx_i=c_{ij}x_ix_j+d_{ij}\colon1\leq i<j\leq n\}\rangle
	\]
	is called a \emph{$G$-algebra}, if $\{x_1^{\alpha_1}\cdot\ldots\cdot x_n^{\alpha_n}:\alpha_i\in\mathbb{N}_0\}$ is a $K$-basis of $A$.
\end{definition}
$G$-algebras \cite{LS, lev_diss} are also known as \emph{algebras of solvable type} \cite{KW,Kredel2015,Kr} and as \emph{PBW algebras} \cite{BGV}. $G$-algebras are left and right Noetherian domains that occur naturally in various situations and encompass algebras of linear functional operators modeling difference and differential equations.

\begin{example}
\label{exAlgebrasOfOps}
	Let $K$ be a field, $q_i\in K\setminus\{0\}$ and $n\in\mathbb{N}$.
	Common $G$-algebras include the following examples, where only the relations between non-commutating variables are listed:
	\begin{itemize}
		\item
			The commutative polynomial ring $K[x_1,\ldots,x_n]$.
		\item
			The $n$-th \emph{Weyl algebra} $A_n:=K\langle x_1,\ldots,x_n,\partial_1,\ldots,\partial_n\rangle$ with $\partial_ix_i=x_i\partial_i+1$ for all $1\leq i\leq n$.
		\item
			The $n$-th \emph{shift algebra} $S_n:=K\langle x_1,\ldots,x_n,s_1,\ldots,s_n\rangle$ with $s_ix_i=x_is_i+s_i=(x_i+1)s_i$ for all $1\leq i\leq n$.
		\item
			The $n$-th \emph{$q$-shift algebra} $S^{(q)}_n:=K\langle x_1,\ldots,x_n,s_1,\ldots,s_n\rangle$ with $s_ix_i=q_i x_is_i$ for all $1\leq i\leq n$.
		\item
			The $n$-th \emph{$q$-Weyl algebra} $A^{(q)}_n:=K\langle x_1,\ldots,x_n,\partial_1,\ldots,\partial_n\rangle$ with $\partial_ix_i=q_i x_i\partial_i+1$ for all $1\leq i\leq n$.
		\item
			The $n$-th \emph{integration algebra} $K\langle x_1,\ldots, x_n, I_1,\ldots, I_n\rangle$
			with $I_i x_i = x_i I_i + I_i^2$  for all $1\leq i\leq n$.
	\end{itemize}
\end{example}

Furthermore, there exists a well-developed Gr\"obner basis theory for $G$-algebras which is close to the commutative case and not only allows us to explicitly compute $\ker(\varphi_{s,r})$ (which is finitely generated), but also to solve the equation $\tilde{s}r=\tilde{r}s$ for $\tilde{r}$ via division with remainder. Details can be found in \cite{lev_diss}.

Note that the basic concept of \textsc{LeftOre} can be adapted to other, yet more general settings.

\subsection{A partial classification of Ore localizations}

The only remaining problem to solve is the intersection of a left ideal with a left Ore set.
Unfortunately, due to their multiplicative nature, left Ore sets are seldom finitely generated as monoids, therefore we have to single out interesting classes of Ore sets and deal with them individually.
To this end, we propose the following partial classification of left Ore localizations:

\begin{definition}
	Let $K$ be a field and $R$ a $K$-algebra and a domain.
	\begin{itemize}
		\item
			Let $S$ be a left Ore set in $R$ that is generated as a multiplicative monoid by at most countably many elements.
			Then $S^{-1}R$ is called a \emph{monoidal localization}.
		\item
			Let $n\in\mathbb{N}$, $K[x]:=K[x_1,\ldots,x_n]$ a subring of $R$ and $\mathfrak{p}$ a prime ideal in $K[x]$, then $S:=K[x]\setminus\mathfrak{p}$ is multiplicatively closed.
			If $S$ is a left Ore set in $R$, then $S^{-1}R$ is called a \emph{geometric localization}.
		\item
			Let $T$ be a $K$-subalgebra of $R$, then $S:=T\setminus\{0\}$ is multiplicatively closed.
			If $S$ is a left Ore set in $R$, then $S^{-1}R$ is called a \emph{(partial) rational localization}.
	\end{itemize}
\end{definition}

All three types of localizations have commutative counterparts:

\begin{example}
	Let $R$ be a commutative domain and $K$ a field.
	\begin{itemize}
		\item
			Let $f\in R\setminus\{0\}$, then $R_f=[f]^{-1}R=\{f^k\mid k\in\mathbb{N}_0\}^{-1}R$ is a monoidal localization.
		\item
			Let $\mathfrak{p}$ be a prime ideal in the polynomial ring $K[x]$, then $K[x]_\mathfrak{p}=(K[x]\setminus\mathfrak{p})^{-1}K[x]$ is a geometric localization.
		\item
			$\operatorname{Quot}(R)=(R\setminus\{0\})^{-1}R$ is a rational localization, as well as $K(x)[y]=(K[x]\setminus\{0\})^{-1}K[x,y]$.
	\end{itemize}
\end{example}

An important instance of rational localization is the following generalization of the classical quotient field construction:

\begin{definition}
	If $S:=R\setminus\{0\}$ is a left Ore set in a domain $R$, then $R$ is called a \emph{left Ore domain}.
	The localization $S^{-1}R$ is called \emph{left quotient (skew) field} of $R$ and denoted $\operatorname{Quot}(R)$.
\end{definition}

Therefore, any left Ore domain can be embedded into a division ring.
This holds in particular for any $G$-algebra.

\section{Case study: localizations of the first Weyl algebra}\label{case_study_Weyl}

In this section we want to explore at an example how one can utilize left saturation not only to gain theoretical insight but also as a preprocessing step before attempting any computations in a computer algebra system.
To this end we will consider a field $K$ and $A_1=A_1(K)=K\langle x,\partial\mid\partial x=x\partial+1\rangle$, the first Weyl algebra over $K$.

We are now interested in making $x$ and $\partial$ invertible, which means finding a left Ore set in $A_1$ that contains $V:=[x,\partial]$.
By the forthcoming \Cref{Weyl_Ore_sets} we have that $[x]$ and $[\partial]$ already are left Ore sets in $A_1$, thus $V:=[[x]\cup[\partial]]$ is indeed a left Ore set in $A_1$ as a multiplicatively closed set generated by left Ore sets (Lemma 4.1 in \cite{krause_lenagan}).

The \emph{Euler operator} in $A_1$ is defined as $\theta:=x\cdot\partial\in A_1$.
The following ``commutation rules'' can be proven by induction:

\begin{lemma}\label{theta_commutation_rules}
	For all $m,n\in\mathbb{N}_0$ and all $z\in K$ we have
	\[
		(\theta+z)^mx^n
		=x^n(\theta+z+n)^m
		\quad\text{and}\quad
		\partial^n(\theta+z)^m
		=(\theta+z+n)^m\partial^n.
	\]
\end{lemma}

Now we are able to compute the left saturation closure of $V$ with some additional knowledge about factorizations:

\begin{proposition}
	Let $p:=\operatorname{char}(K)\in\mathbb{P}\cup\{0\}$.
	Then
	\[
		\operatorname{LSat}(V)
		=[\{x,\partial\}\cup(\theta+((\mathbb{Z}/p\mathbb{Z})\setminus\{0,1\}))\cup(K\setminus\{0\})].
	\]
\end{proposition}
\begin{proof}
	Let $w\in\mathbb{Z}/p\mathbb{Z}$ and $S$ be the right-hand side of the equation in the claim.
	First consider the case $p=0$: if $w\in\mathbb{N}_0$, then $x^w(\theta+w)=\theta x^w=x\partial x^w\in V$, if $-w\in\mathbb{N}$, then $\partial^{-w}(\theta+w)=(\theta+w-w)\partial^{-w}=\theta\partial^{-w}=x\partial^{1-w}\in V$.
	If $p>0$ we can always find $n\in\mathbb{N}_0$ such that $w=n+p\mathbb{Z}$ and we get $x^n(\theta+w)=\theta x^n\in V$.
	In any case we see that $\theta+w\in\operatorname{LSat}(V)$.
	Furthermore, for any $k\in K\setminus\{0\}$ we have $k^{-1}\cdot k=1\in V$, thus $S\subseteq\operatorname{LSat}(V)$.\\
	Now let $r\in\operatorname{LSat}(V)$, then there exists $w\in A_1$ such that $wr\in V$, thus $wr=\prod_{i=1}^{n}s_i$ for some $n\in\mathbb{N}_0$ and $s_i\in\{x,\partial\}$.
	Using the commutation rules from \Cref{theta_commutation_rules} we can rewrite $wr$ into $wr=p\cdot q\cdot t^m$, where $t\in\{x,\partial\}$, $m\in\mathbb{N}_0$, $p\in[\theta+((\mathbb{Z}/p\mathbb{Z})\setminus\{0,1\})]$ and $q\in[\theta,\theta+1]$ (note that $\theta+w_1$ and $\theta+w_2$ commute for $w_1,w_2\in K$).
	According to Lemma 2.6 in \cite{GHL15p} any other non-trivial factorization of $wr$ (like $wr$ itself) can be derived by using the commutation rules and rewriting $\theta$ resp. $\theta+1$ as $x\partial$ resp. $\partial x$.
	But all factors that can be created in this way are already contained in $S$, thus $r\in S$ and therefore $\operatorname{LSat}(V)\subseteq S$.
\end{proof}

We will now see that this localization behaves fundamentally different depending on $p$.
We need the notion of the Gelfand-Kirillov dimension ($\operatorname{GKdim}$, see \cite{krause_lenagan,mcconnell_robson}), which is defined for both rings and modules with respect to a fixed field $K$. Over Noetherian domains its behavior is somewhat similar to that of Krull dimension ($\operatorname{Krdim}$) over commutative rings.

Note that for any field $K$ one has $\operatorname{GKdim}_K(A_1(K))=2$.
%regardless of $p$ 

\subsection{Characteristic zero case}

%In characteristic zero 

\begin{lemma}
	Let $K$ be an algebraically closed field of characteristic $p=0$.
	Then %we have the following:
	\begin{enumerate}[(a)]
		\item
			$\operatorname{GKdim}(V^{-1}A_1)=3$,
		\item
			$\operatorname{GKdim}(\operatorname{Quot}(A_1))=\infty$.
	\end{enumerate}
\end{lemma}
\begin{proof}
	The first claim follows from Example 4.11 in \cite{krause_lenagan}, while the second result is due to Makar-Limanov, who showed in \cite{makar-limanov} that $\operatorname{Quot}(A_1)$ contains a free subalgebra generated by two elements.
\end{proof}

In particular, we have $\operatorname{GKdim}(A_1)<\operatorname{GKdim}(V^{-1}A_1)<\operatorname{GKdim}(\operatorname{Quot}(A_1))$.

\subsection{Positive characteristic case}

In positive characteristic the Weyl algebra has center $K[x^p,\partial^p]$, while in the case of
characteristic zero the center is just $K$.

\begin{lemma}
Let $K$ be of characteristic $p>0$. Then 
%we have the following:
%, with the notations as above,
%	Then $\operatorname{GKdim}(V^{-1}A_1)=2$.
%\end{lemma}
		\begin{enumerate}[(a)]
		\item
			$\operatorname{GKdim}(V^{-1}A_1)=2$,
		\item
			$\operatorname{GKdim}(\operatorname{Quot}(A_1))=2$.
	\end{enumerate}
	In particular, any localization of $A_1$ has Gelfand-Kirillov dimension 2.
\end{lemma}

\begin{proof}
Though the first claim follows from the second, we give a direct proof of it, which illustrates an important technique.
Let $T:=[x^p,\partial^p]$, then clearly $\operatorname{LSat}(V)=\operatorname{LSat}([x,\partial])=\operatorname{LSat}([x^p,\partial^p])=\operatorname{LSat}(T)$.
Since $T$ is contained in the center of $A_1$ it is also a left Ore set in $A_1$, thus $V^{-1}A_1\cong T^{-1}A_1$ by \Cref{LSat_equality_implies_localizations_isomorphic}.
By Proposition 4.2 in \cite{krause_lenagan} the Gelfand-Kirillov dimension does not change when passing from $A_1$ to a central localization of $A_1$ like $T^{-1}A_1$, which implies
\[
	\operatorname{GKdim}(V^{-1}A_1)
	=\operatorname{GKdim}(T^{-1}A_1)
	=\operatorname{GKdim}(A_1)
	=2.
\]
Now we proceed with the second claim.
Consider $Z=K[x^p,\partial^p]$, the center of $A_1(K)$, then $\operatorname{GKdim}(Z)=\operatorname{Krdim}(Z)=2$ by Corollary 4.4  in \cite{krause_lenagan} since $Z$ is a commutative domain.
Moreover, the latter also implies that $Z\setminus\{0\}$ is an Ore set in $A_1(K)$.
We claim that $(Z\setminus\{0\})^{-1}A_1 = (A_1\setminus\{0\})^{-1}A_1 = \operatorname{Quot}(A_1)$, 
in other words, $\operatorname{LSat}(Z\setminus\{0\}) = A_1\setminus\{0\}$.
It is enough to show that any left ideal $\{0\} \neq L \subset A_1$ has a non-zero intersection with $Z$.
Suppose that there is $L$ such that $L \cap Z = \{0\}$.
By Lemma 8.5 in \cite{krause_lenagan} it follows that then $\operatorname{GKdim}(A_1/L) \geq \operatorname{GKdim}(Z)$ holds.
Thus $2 = \operatorname{GKdim}(A_1) \geq \operatorname{GKdim}(A_1/L) \geq 2$ and therefore by Corollary 8.6  in \cite{krause_lenagan} $L=\{0\}$ follows. 
\end{proof}

\section{Computing the intersection of a left ideal and a left Ore set}\label{section_intersection}

This section provides the theory and algorithms to compute a representation of the intersection of a left ideal $I$ with a left Ore set $S$, where $S$ belongs to one of the three types stated above, within the setting of a $G$-algebra $A$ over a field $K$.
To avoid rather trivial cases we will assume that $I$ is neither the zero ideal nor the whole algebra, i.e. $\{0\}\subsetneq I\subsetneq A$.

\subsection{Monoidal localizations}

Monoidal localization allows us to adjoin inverses of certain elements, which for example describes the transition from the polynomial ring $K[x]$ to the Laurent polynomial ring $K[x,x^{-1}]$.

For now, let $S=[g]$ for some $g\in A\setminus K$.
To the best of our knowledge it is not possible in general to decide whether $I\cap S$ is empty, but in some cases we can give a positive answer:

\begin{remark}\label{intersection_with_monoid_algebra}
	%Let $M$ be a submonoid of $A\setminus\{0\}$.
	Since $A$ is a domain so is $K[S]=K[g]\subseteq A$, the $K$-monoid algebra of $S$.
	%Let $L$ be a left ideal in $A$, then $L\cap M\subsetneq L\cap K[M]=:L'$.
	Assume that we are able to compute $L:=I\cap K[S]$\footnote{See \cite{LVint} for conditions and further details.}.
	If $L=\{0\}$, then in particular $I\cap S=\emptyset$ since $I\cap S=L\cap S$.
\end{remark}

In the following we will assume $I\cap S\neq\emptyset$.
For our purposes this is not a restriction since we are mostly interested in the case where $I=\ker(\varphi_{s,r})$ for some $s\in S$ and $r\in R$.
Then $I\cap S\neq\emptyset$ follows from \Cref{Ore_condition_iff_all_kernels_intersect}.

If $I\cap S$ is non-empty it has the structure of a principal monoid ideal:

\begin{lemma}
	Let $g\in A\setminus K$.
	If $I\cap[g]\neq\emptyset$, then there exists $m\in\mathbb{N}$ such that $I\cap[g]=g^m\cdot[g]$.
\end{lemma}
\begin{proof}
	If $I\cap[g]\neq\emptyset$, then there exists a minimal $m\in\mathbb{N}$ such that $g^m\in I$ and $g^k\notin I$ for all $k<m$.
	Now $g^j=g^{j-m}\cdot g^m\in I$ for any $j\geq m$, thus $I\cap[g]=g^m\cdot[g]$.
\end{proof}

The natural thing to do is to iterate over the natural numbers to find the smallest $m$ among them such that $g^m\in I$.
This membership test can be done by computing normal forms in the Gr\"obner sense: in the algorithm \textsc{MonoidalIntersection}, $\operatorname{NF}(g^m|F)$ denotes the normal form of $g^m$ with respect to $F$.

To avoid unnecessary expensive normal form computations we use another fact from Gr\"obner basis theory: the leading monomial of any element of $I$ must be divisible by the leading monomial of an element of a Gr\"obner basis of $I$.
Since we know that at least one power of $g$ is contained in $I$, we find the minimal $m$ such that $\operatorname{lm}(g^m)$ is divisible by the leading monomial of any basis element, which can be done by comparing the leading exponents.

\begin{algorithm}
	\KwIn{Gr\"obner basis $F=\{f_1,\ldots,f_k\}$ of $I$ and $g\in R\setminus K$ with $I\cap[g]\neq\emptyset$.}
	\KwOut{$m\in\mathbb{N}$ such that $I\cap[g]=g^m\cdot[g]$.}
	\Begin{
		\ForEach{$1\leq i\leq k$}{
			$m_i:=\min\{k\in(\mathbb{N}\cup\{\infty\}):\operatorname{lm}(f_i)|\operatorname{lm}(g)^k\}$\;
		}
		$m:=\min\{m_i\mid1\leq i\leq k\}$\;
		\While{$\operatorname{NF}(g^m|F)\neq0$}{
			$m:=m+1$\;
		}
		\Return{$m$}\;
	}
	\caption{\textsc{MonoidalIntersection}}
\end{algorithm}

While multiplicatively closed sets generated by infinitely many elements are out of scope for computational purposes, we still have to deal with finite sets of generators.
To reduce this to the case of one generator, we generalize the following classical result: let $f_1,\ldots,f_k\in K[\underline{x}]:=K[x_1,\ldots,x_n]$, $S:=[f_1,\ldots,f_k]$ and $T=[f_1\cdot\ldots\cdot f_k]$, then $S^{-1}K[\underline{x}]\cong T^{-1}K[\underline{x}]$.

\begin{lemma}\label{localization_at_product_is_isomorphic_to_localization_at_factors}
	Let $R$ be a domain and $g_1,\ldots,g_k\in R\setminus\{0\}$ such that $g_ig_j=g_jg_i$ for all $i$ and $j$.
	Consider $S=[g_1,\ldots,g_k]$ and $T=[g]$ for $g:=g_1\cdot\ldots\cdot g_k$.
	\begin{enumerate}[(a)]
		\item
			$S$ is a left Ore set in $R$ if and only if $T$ is a left Ore set in $R$.
		\item
			If $S$ and $T$ are left Ore sets, then $S^{-1}R\cong T^{-1}R$.
	\end{enumerate}
\end{lemma}
\begin{proof}
	By construction, $S$ and $T$ are multiplicatively closed sets such that $T\subseteq S$.
	Since the $g_i$ commute we have
	\[
		(g_1\cdot\ldots\cdot g_{j-1}\cdot g_{j+1}\cdot\ldots\cdot g_k)g_j=g\in T
	\]
	which implies $g_j\in\operatorname{LSat}(T)$ and thus $S\subseteq\operatorname{LSat}(T)$, since the $g_j$ generate $S$ as a monoid.
	Together with $T\subseteq S\subseteq\operatorname{LSat}(S)$ we get $\operatorname{LSat}(S)=\operatorname{LSat}(T)$ by applying Lemma \ref{S_in_LSat(T)_iff_LSat(S)_in_LSat(T)} twice.
	Now the first part follows from Lemma \ref{S_Ore_iff_LSat(S)_Ore}, the second from Proposition \ref{localization_at_left_saturation}.
\end{proof}

\subsection{Geometric localizations of Weyl algebras}

Let $n\in\mathbb{N}$ and $\mathfrak{p}$ be a prime ideal in $R:=K[x_1,\ldots,x_n]\subsetneq A_n$.
Then $R\setminus\mathfrak{p}$ is a left Ore set in $A_n$ and we can consider the geometric localization $(R\setminus\mathfrak{p})^{-1}A_n$.
The most common occurrence of this localization is the special case where we replace $\mathfrak{p}$ by the maximal ideal $\mathfrak{m}_p$ in $R$ corresponding to a point $p\in K^n$.
The result is the so-called \emph{local (algebraic) Weyl algebra} $A_{n,p}:=(R\setminus\mathfrak{m}_p)^{-1}A_n$, which is important in $D$-module theory.

%Let $n\in\mathbb{N}$ and $\mathfrak{m}_p$ be the maximal ideal in $R:=K[x_1,\ldots,x_n]$ corresponding to a point $p\in K^n$.
%The most common occurrence of geometric localization is the so-called \emph{local (algebraic) Weyl algebra} $A_{n,p}:=(R\setminus\mathfrak{m}_p)^{-1}A_n$, which is important in $D$-module theory.

The main theoretical result in this paragraph is that the Weyl algebras contain a multitude of left Ore sets.
To prove this we first need some technical results.
Note that due to the relations in $A_n$ we have $f\partial_j=\partial_jf+\frac{\partial f}{\partial x_j}$ for all $f\in R$.

\begin{lemma}
	Let $f\in R$ and $j\in\{1,\ldots,n\}$.
	For all $i\in\mathbb{N}_0$ we have
	\[
		f^{i+1}\partial_j
		=\left(\partial_jf-(i+1)\frac{\partial f}{\partial x_j}\right)f^i.
	\]
\end{lemma}
\begin{proof}
	Induction on $i\in\mathbb{N}_0$:
	%\begin{itemize}
		%\item[(IB)]
			let $i=0$, then $f^1\partial_j=f\partial_j=\partial_jf-\frac{\partial f}{\partial x_j}=\left(\partial_jf-1\cdot\frac{\partial f}{\partial x_j}\right)f^0$.
		%\item[(IS)]
			%$i\rightsquigarrow i+1$:
			Assume that the claim holds for $i\in\mathbb{N}_0$, then we have
			%We have
			\[\begin{split}
				f^{i+2}\partial_j
				&=ff^{i+1}\partial_j
				%\overset{\text{(IH)}}{=}
				=f\left(\partial_jf-(i+1)\frac{\partial f}{\partial x_j}\right)f^i
				=\left(f\partial_jf-(i+1)f\frac{\partial f}{\partial x_j}\right)f^i\\
				&=\left(f\partial_j-(i+1)\frac{\partial f}{\partial x_j}\right)f^{i+1}
				%\overset{\text{(IB)}}{=}
				\left(\partial_jf-\frac{\partial f}{\partial x_j}-(i+1)\frac{\partial f}{\partial x_j}\right)f^{i+1}\\
				&=\left(\partial_jf-(i+2)\frac{\partial f}{\partial x_j}\right)f^{i+1}.
			\end{split}\]
	%\end{itemize}
\end{proof}

\begin{lemma}\label{Weyl_monoidal_Ore_sets_prelim}
	Let $f\in R$, $i\in\mathbb{N}_0$ and $\beta\in\mathbb{N}_0^n$ such that $i+1\geq\abs{\beta}$.
	Then there exists $v_{i+1,\beta}\in A_n$ such that
	\begin{enumerate}[(i)]
		\item
			$\operatorname{tdeg}_\partial(v_{i+1,\beta})<\abs{\beta}$,
		\item
			$v_{i+1,\beta}$ only contains partial derivatives of $f$ of the form $\frac{\partial^{\abs{\alpha}}f}{\partial x^\alpha}$, where $\beta-\alpha\in\mathbb{N}_0^n$ and
		\item
			$f^{i+1}\partial^\beta=(\partial^\beta f^{\abs{\beta}}+v_{i+1,\beta})f^{i+1-\abs{\beta}}$.
	\end{enumerate}
\end{lemma}
\begin{proof}
	Induction on $\abs{\beta}\in\mathbb{N}_0$:
	%\begin{itemize}
		%\item[(IB)]
			if $\abs{\beta}=0$, then $\beta=0$.
			Set $v_{i+1,0}:=0$, then
			\[
				f^{i+1}\partial^\beta=f^{i+1}=(\partial^\beta f^{\abs{\beta}}+v_{i+1,0})f^{i+1-\abs{\beta}}.
			\]
		%\item[(IS)]
			Now let $\beta\in\mathbb{N}_0^n\setminus\{0\}$ and assume the claim holds for all $\alpha\in\mathbb{N}_0^n$ with $\abs{\alpha}<\abs{\beta}$.
			Then $\beta=\alpha+e_j$ for some $j\in\{1,\ldots,n\}$ and $\alpha\in\mathbb{N}_0^n$ with $\abs{\alpha}<\abs{\beta}$.
			Now
			\[\begin{split}
				f^{i+1}\partial^\beta
				&=f^{i+1}\partial^\alpha\partial_j
				%\overset{\text{(IH)}}{=}
				(\partial^\alpha f^{\abs{\alpha}}+v_{i+1,\alpha})f^{i+1-\abs{\alpha}}\partial_j
				=\partial^\alpha f^{i+1}\partial_j+v_{i+1,\alpha}f^{i+1-\abs{\alpha}}\partial_j\\
				&=\partial^\alpha\left(\partial_jf-(i+1)\frac{\partial f}{\partial x_j}\right)f^i
					+v_{i+1,\alpha}\left(\partial_jf-(i+1-\abs{\alpha})\frac{\partial f}{\partial x_j}\right)f^{i-\abs{\alpha}}\\
				&=\left(\partial^\alpha\partial_jf^{1+\abs{\alpha}}-(i+1)\partial^\alpha\frac{\partial f}{\partial x_j}f^{\abs{\alpha}}+v_{i+1,\alpha}\partial_jf-(i+1-\abs{\alpha})v_{i+1,\alpha}\frac{\partial f}{\partial x_j}\right)f^{i-\abs{\alpha}}\\
				&=(\partial^\beta f^{\abs{\beta}}+v_{i+1,\beta})f^{i+1-\abs{\beta}},
			\end{split}\]
			where $v_{i+1,\beta}:=-(i+1)\partial^\alpha\frac{\partial f}{\partial x_j}f^{\abs{\alpha}}+v_{i+1,\alpha}\partial_jf-(i+1-\abs{\alpha})v_{i+1,\alpha}\frac{\partial f}{\partial x_j}$ satisfies the conditions above.
	%\end{itemize}
\end{proof}

\begin{lemma}\label{Weyl_monoidal_Ore_sets}
	Let $f\in R$, $r\in A_n$, $d:=\operatorname{tdeg}_\partial(r)$ and $k\in\mathbb{N}_0$.
	Then there exist $\tilde{r},\hat{r}\in A_n$ such that
	\[
		f^{d+k}\cdot r
		=\tilde{r}\cdot f^k
		\quad\text{and}\quad
		r\cdot f^{d+k}
		=f^k\cdot\hat{r}.
	\]
	%In particular, the set $\set{f^k\mid k\in\IN_0}$ is a left Ore set in $A_n$.
\end{lemma}
\begin{proof}
	Let $r=\sum_{\beta\in\mathbb{N}_0^n}^{}b_\beta\partial^\beta$, where $b_\beta\in R$.
	Since $d+k\geq\abs{\beta}$ for all $\beta$ such that $b_\beta\neq0$, by \Cref{Weyl_monoidal_Ore_sets_prelim} there exist $v_{d+k,\beta}\in A_n$ such that
	\[
		f^{d+k}\partial^\beta
		=(\partial^\beta f^{\abs{\beta}}+v_{d+k,\beta})f^{d+k-\abs{\beta}}
		=(\partial^\beta f^{\abs{\beta}}+v_{d+k,\beta})f^{d-\abs{\beta}}f^k.
	\]
	Define
	\[
		\tilde{r}
		:=\sum_{\beta\in\mathbb{N}_0^n}^{}b_\beta(\partial^\beta f^{\abs{\beta}}+v_{d+k,\beta})f^{d-\abs{\beta}},
	\]
	then
	\[
		f^{d+k}\cdot r
		=f^{d+k}\cdot\sum_{\beta\in\mathbb{N}_0^n}^{}b_\beta\partial^\beta
		=\sum_{\beta\in\mathbb{N}_0^n}^{}b_\beta f^{d+k}\partial^\beta
		=\sum_{\beta\in\mathbb{N}_0^n}^{}b_\beta(\partial^\beta f^{\abs{\beta}}+v_{d+k,\beta})f^{d-\abs{\beta}}f^k
		=\tilde{r}\cdot f^k.
	\]
	The other statement can be shown analogously using a right-sided version of \Cref{Weyl_monoidal_Ore_sets_prelim}.
\end{proof}

%A rather tedious technical verification yields the following result:
%
%\begin{lemma}
%	Let $f\in K[x_1,\ldots,x_n]$, $r\in A_n$ and $k\in\mathbb{N}_0$.
%	Further, let $d\in\mathbb{N}_0$ be the \emph{order}\footnote{i.e. the total degree with respect to the variables $\partial_1,\ldots,\partial_n$, also denoted by $\operatorname{tdeg}_\partial$.} of $r$. 
%		Then there exist $\tilde{r},\hat{r}\in A_n$ such that
%	\[
%		f^{d+k}\cdot r=\tilde{r}\cdot f^k
%		\quad\text{and}\quad
%		r\cdot f^{d+k}=f^k\cdot\hat{r}.
%	\]
%\end{lemma}

\begin{proposition}\label{Weyl_Ore_sets}
	Let $S$ be a multiplicatively closed set in $R=K[x_1,\ldots,x_n]$ and $T$ a multiplicatively closed set in $K[\partial_1,\ldots,\partial_n]$.
	Then $S$ and $T$ are left and right Ore sets in $A_n$.
\end{proposition}
\begin{proof}
	Since $S$ is a multiplicatively closed set in $R$ it is also a multiplicatively closed set in $A_n$.
	Let $r\in A_n$ and $s\in S$.
	By \Cref{Weyl_monoidal_Ore_sets} there exist $\tilde{r},\hat{r}\in A_n$ such that $s^{d+1}\cdot r=\tilde{r}\cdot s$ and $r\cdot s^{d+1}=s\cdot\hat{r}$, where $d:=\operatorname{tdeg}_\partial(r)$.
	Since $S$ is multiplicatively closed we have $s^{d+1}\in S$, therefore $S$ satisfies the left and the right Ore condition in $A_n$, thus $S$ is a left and right Ore set in $A_n$.
	The statement for $T$ follows from analogous calculations.
\end{proof}

This implies that any multiplicatively closed set in $R$ is a left and right Ore set in $A_n$, in particular we have that geometric localization of $A_n$ is possible at the complement of any prime ideal $\mathfrak{p}$ in $R$.
But even in closely related $G$-algebras like the shift algebra this does not need to hold, as the following example demonstrates:

\begin{example}\label{example_geometric_shift}
	Consider the prime ideal $\mathfrak{p}=\langle x+1\rangle$ in $K[x]\subseteq S_1$, then for the pair $(x,s)\in(K[x]\setminus\mathfrak{p})\times S_1$ a simple computation delivers $\ker(\varphi_{x,s})=\langle x+1\rangle=\mathfrak{p}$.
	Therefore $\ker(\varphi_{x,s})\cap S=\emptyset$, so $K[x]\setminus\mathfrak{p}$ is not a left Ore set in $S_1$ by Proposition \ref{Ore_condition_iff_all_kernels_intersect}.
\end{example}

Thus the main application of the geometric type is localizing the $n$-th Weyl algebra $A_n$ at the left Ore set $S:=R\setminus\mathfrak{p}$, where $\mathfrak{p}$ is a prime ideal in $R$.

In contrast to the two other types of localizations, the intersection of $I$ and $S$ has no exploitable additional structure: while it is a multiplicatively closed set without $1$, it need not be finitely generated.

Therefore, in the algorithm \textsc{GeometricIntersection}, we return essentially the intersection of $I$ and $R$, which can be computed via Gr\"obner-driven elimination of variables.

\begin{algorithm}
	\KwIn{A left ideal $I$ of $A_n$, $S$, $R$ and $\mathfrak{p}$ as above.}
	\KwOut{A left ideal $J$ in $R$ such that $I\cap S=J\setminus\mathfrak{p}$.}
	\Begin{
		compute $\tilde{I}:=I\cap R=\langle m_1,\ldots,m_k\rangle$\;
		\ForEach{$1\leq i\leq k$}{
			let $\tilde{m_i}$ be the normal form of $m_i$ with respect to $\mathfrak{p}$\;
		}
		\Return{$J:=\langle\tilde{m_1},\ldots,\tilde{m_k}\rangle$}\;
	}
	\caption{\textsc{GeometricIntersection}}
\end{algorithm}

An element $f\in I\cap R$ is an element of $I\cap S$ if and only if $f\notin\mathfrak{p}$, which can be checked by computing the normal form of $f$ with respect to $\mathfrak{p}$.

\begin{proposition}
	In the situation of the algorithm \textsc{GeometricIntersection}, $I\cap S=\emptyset$ if and only if $\tilde{m_i}=0$ for all $i$.
\end{proposition}
\begin{proof}
	By construction, $\tilde{m_i}=0$ for all $i$ if and only if $m_i\in\mathfrak{p}$ for all $i$, which is equivalent to $I\cap R\subseteq\mathfrak{p}$.
	From
	\[
		I\cap R
		=I\cap((R\setminus\mathfrak{p})\cup\mathfrak{p})
		=I\cap(S\cup\mathfrak{p})
		=(I\cap S)\cup(I\cap\mathfrak{p})
	\]
	we can see that this is equivalent to $I\cap S=\emptyset$, since $I\cap S\subseteq R\setminus\mathfrak{p}$.
\end{proof}

\noindent
Thus, if $I\cap S\neq\emptyset$, a member of this intersection can be found among the non-zero generators of $J$.

\subsection{Rational localizations}

In algebras of linear operators, rational localization provides the formal mechanism of passing from polynomial to rational coefficients, for example from the polynomial Weyl algebra $A_1$ to the first rational Weyl algebra $(K[x]\setminus\{0\})^{-1}A_1$.

To set the scene, let $A$ be generated as a $G$-algebra by the variables $x_1,\ldots,x_n$ and let $V\subseteq\{1,\ldots,n\}$ such that $\{x_i\mid i\in V\}$ generate a subalgebra $B$ of $A$ and $S:=B\setminus\{0\}$ is a left Ore set in $A$.
If we can eliminate the variables $\{x_i\mid i\in\{1,\ldots,n\}\setminus V\}$ with Gr\"obner-driven elimination\footnote{In contrast to the commutative case, this is not always possible, see \cite{LVint, lev_diss}.}, then the algorithm \textsc{RationalIntersection} computes the intersection of $S$ and $I$.

\begin{algorithm}
	\KwIn{A left ideal $I$ of $A$, $B$ as above.}
	\KwOut{The intersection $I\cap S$.}
	\Begin{
		compute $J:=I\cap B$ via elimination\;
		\Return{$J\setminus\{0\}$}\;
	}
	\caption{\textsc{RationalIntersection}}
\end{algorithm}

\section{Further algorithmic aspects}

\subsection{The right side analogon}

While we concentrate mostly on the left-sided version of non-commutative structures, the right-sided analogues of the given definitions and results hold as well, which can also be seen by considering \emph{opposite structures}:

\begin{definition}
	Let $(R,+,\cdot)$ be a ring, then the \emph{opposite ring} of $R$ is $R^{\text{op}}:=(R,+,*)$, where $a*b:=b\cdot a$ for all $a,b\in R$.
\end{definition}

In particular, a right Ore set in $R$ is a left Ore set in $R^{\text{op}}$.
Most algorithms for non-commutative structures in \textsc{Singular:Plural} are only implemented for the left-sided versions, while right-sided computations are carried out in a left-sided setting in the opposite ring.
Note that there are special tools for handling opposite rings and the process of creating opposite objects.

\subsection{The left-right conundrum}

Another classical result in the theory of Ore localization is the following: if a multiplicative subset $S$ of a domain $R$ is both left and right Ore, then the left Ore localization $S^{-1}R$ is isomorphic to the right Ore localization $RS^{-1}$ via
\[
	RS^{-1}\rightarrow S^{-1}R,\quad rs^{-1}\mapsto\tilde{s}^{-1}\tilde{r},
\]
where $\tilde{s}r=\tilde{r}s$.
Given a right fraction $rs^{-1}\in RS^{-1}$, finding a corresponding left fraction $\tilde{s}^{-1}\tilde{r}\in S^{-1}R$ is therefore just another application of the left Ore condition, while computing the inverse image of a left fraction requires the right Ore condition.

\subsection{Basic arithmetic}

If we examine Theorem \ref{construction_of_Ore_localization} closely we can see that addition and multiplication in $S^{-1}R$ only consist of computing one instance of the left Ore condition as well as some basic additions and multiplications in the base ring $R$, which directly gives us algorithms for addition and multiplication.

\subsection{Computing inverses}\label{section_computing_inverse}

Additive inverses are given by $-(s,r)=(s,-r)$, but, as we have seen earlier, multiplicative inverses are immensely more complicated.
Proposition \ref{fraction_invertible_iff_numerator_in_LSat} tells us that a fraction $(s,r)$ is invertible if and only if $r\in\operatorname{LSat}(S)$, thus deciding invertibility of a fraction is not harder than computing $\operatorname{LSat}(S)$.
After Remark \ref{rem:repr} we do not expect $\operatorname{LSat}(S)$ to be presentable in finite terms.
However, in the case of geometric localizations at a prime ideal $\mathfrak{p}$ we are in the fortunate situation that $S$ is already saturated:

\begin{lemma}
	Let $\mathfrak{p}$ be a prime ideal in $R:=K[x_1,\ldots,x_n]$.
	Then both $R\setminus\{0\}$ and $S:=R\setminus\mathfrak{p}$ are saturated in $A_n$.
\end{lemma}
\begin{proof}
	For the first part, consider a global monomial ordering where $\partial_i>x_j$ for all $i,j$.
	Let $a,b\in A_n$ such that $a\cdot b\in R\setminus\{0\}$, then $a$ and $b$ are non-zero, thus $\operatorname{tdeg}_\partial(a)+\operatorname{tdeg}_\partial(b)=\operatorname{tdeg}_\partial(a\cdot b)=0$.
	Therefore both $a$ and $b$ are contained in $R\setminus\{0\}$.
	
	Now let $a\cdot b\in S\subseteq R\setminus\{0\}$, then by the first part we have $a,b\in R\setminus\{0\}$.
	Since $a\in\mathfrak{p}$ and $b\in\mathfrak{p}$ both imply $a\cdot b\in\mathfrak{p}$, we have $a,b\in R\setminus\mathfrak{p}=S$.
\end{proof}

Thus we have that a fraction $(s,r)$ in a geometric localization of $A_n$ at $\mathfrak{p}$ is invertible if and only if $r\in K[x_1,\ldots,x_n]\setminus\mathfrak{p}$; the latter condition can be checked algorithmically with commutative Gr\"obner methods.

Unfortunately, $S$ will not be saturated in general when we consider the other localization types.
A closer look at the definition of left saturation closure yields the following insight:

\begin{lemma}
	Let $S$ be a left Ore set in a domain $R$ and $r\in R$.
	Then $r\in\operatorname{LSat}(S)$ if and only if $Rr\cap S\neq\emptyset$.
\end{lemma}
Therefore we can decide invertibility of a given fraction if we can decide non-emptiness of the intersection of $S$ with a principal left ideal.
For rational localizations this can be checked with the usual Gr\"obner tools, but for monoidal localizations this is still an open problem as stated before.
Some non-units might be identified with the technique described in \Cref{intersection_with_monoid_algebra}.

%At this point no automatic method to check the invertibility in a general Ore localization is known to us.
%Of course, if the numerator of a fraction either is contained in $S$ or is a unit in $R$, then the fraction is definitely invertible. 

%We remark that due to Proposition \ref{fraction_invertible_iff_numerator_in_LSat}, the invertibility of a fraction $(s,r)\in S^{-1}R$ is equivalent to the membership problem of $r$ in $\operatorname{LSat}(S)$.
%Therefore it is not harder than the determination of $\operatorname{LSat}(S)$ itself.

\subsection{Canceling a fraction}

Given a representation $(s,r)$ of a fraction it is a natural question to ask whether there is a simpler representation $(s',r')$ of the same fraction, for example a representation where the total degree of the denominator $s'$ is smaller than the one of $s$.
Canceling a fraction between other computation steps can have a significant impact on the total computation time.
Given that we are not in a unique factorization domain there may be many different representations that we could call simpler than the initial one, thus it is also of interest to find all simpler representations.

We present two approaches to this problem.
% and discuss the trade-off between speed of the computation as well as correctness and completeness of the result.
To this end, let $(s,r)\in S\times R$ be a representation of a fraction in a left Ore localization $S^{-1}R$ of a $G$-algebra $R$.
We want to compute (at least)
\[
	C_{s,r}
	:=\{(\hat{s},\hat{r})\in S\times R\mid\exists f\in R:f\hat{s}=s\text{ and }f\hat{r}=r\},
\]
the set of all representations of $(s,r)$ that can be constructed from $(s,r)$ by left canceling.

%\subsubsection{Commutativity-based canceling}
%
%The first approach is based on ignoring the fact that we are dealing with non-commutativity: compute a commutative greatest common divisor $g$ of $s$ and $r$, then perform a right division of $r$ and $s$ by $g$ to obtain $s'$ and $r'$ respectively.
%If $s'\in S$ and $(s',r')=(s,r)$ in $S^{-1}R$, then we have found a simpler representation.
%
%This rather crude method of course can go very wrong, for example when trying to cancel $(\partial,x\partial)$ in the localization $[\partial]^{-1}A_1$, where $(\partial,x\partial)\neq(1,x)$.
%Thus it is necessary to check the validity of the output.
%
%Furthermore there might be further factors in the result that can be canceled: this method is not able to cancel the factor $\partial-1$ in $(\partial-1,(\partial-1)\cdot(x^2\partial^2+2x\partial))\in[\partial-1]^{-1}A_1$, since $\partial-1$ is not a commutative divisor of $(\partial-1)\cdot(x^2\partial^2+2x\partial)=x^2\partial^3-x^2\partial^2+4x\partial^2-2x\partial+2\partial$.
%
%The main purpose of this method is acting as a preprocessing step and getting rid of some factors in a reasonable short amount of time.

\subsubsection{Syzygy-based canceling}

The first approach is based on computing right and left syzygies (denoted $\operatorname{RSyz}$ resp. $\operatorname{LSyz}$ below), which can be done with Gröbner-driven algorithms.
Let
\[
	M
	:=\operatorname{RSyz}(\begin{bmatrix}s&r\end{bmatrix})=\{\begin{bmatrix}a&b\end{bmatrix}^T\in R^{2\times1}\mid sa+rb=0\}.
\]
Note that $M\neq\{0\}$ since any $G$-algebra is right Noetherian and thus a right Ore domain.
For any $\begin{bmatrix}a&b\end{bmatrix}^T\in M$ let $N_{a,b}:=\operatorname{LSyz}(\begin{bmatrix}a&b\end{bmatrix}^T)=\{\begin{bmatrix}q&p\end{bmatrix}\in R^{1\times2}\mid qa+pb=0\}$. % and $\tilde{N}_{a,b}:=N_{a,b}\cap(S\times R)$.
We also have $N_{a,b}\neq\{0\}$ since $\begin{bmatrix}s&r\end{bmatrix}\in N_{a,b}$.

\begin{lemma}
	Let $\begin{bmatrix}a&b\end{bmatrix}^T\in M\setminus\{0\}$.
	\begin{enumerate}[(a)]
		\item
			We have $C_{s,r}\subseteq N_{a,b}$.
			%Let $(\hat{s},\hat{r})\in C_{s,r}$ and $\begin{bmatrix}a&b\end{bmatrix}^T\in M\setminus\{0\}$.
			%Then $(\hat{s},\hat{r})\in\operatorname{LSyz}(\begin{bmatrix}a&b\end{bmatrix}^T)$.
		\item
			Let $(q,p)\in N_{a,b}$ with $q\in S$.
			Then $(q,p)=(s,r)$ in $S^{-1}R$.
	\end{enumerate}
\end{lemma}
\begin{proof}
	\begin{enumerate}[(a)]
		\item
			Let $\begin{bmatrix}\hat{s}&\hat{r}\end{bmatrix}\in C_{s,r}$, then there exists $f\in R\setminus\{0\}$ such that $s=f\hat{s}$ and $r=f\hat{r}$.
			Since $\begin{bmatrix}a&b\end{bmatrix}\in M$ we have $0=sa+rb=f\hat{s}a+f\hat{r}b=f(\hat{s}a+\hat{r}b)$, which implies $\hat{s}a+\hat{r}b=0$ and thus $\begin{bmatrix}\hat{s}&\hat{r}\end{bmatrix}\in N_{a,b}$.
		\item
			By the left Ore condition on $S$ there exist $\tilde{s}\in S$ and $\tilde{r}\in R$ such that $\tilde{s}q=\tilde{r}s$.
			Now we have $\tilde{s}pb=-\tilde{s}qa=-\tilde{r}sa=\tilde{r}rb$.
			Since $b=0$ would imply the contradiction $a=0$ we can infer that $\tilde{s}p=\tilde{r}r$, which implies $(q,p)=(s,r)$ in $S^{-1}R$.
	\end{enumerate}
\end{proof}

Thus $\tilde{N}_{a,b}:=N_{a,b}\cap(S\times R)$ is a superset of $C_{s,r}$ consisting of representations of $(s,r)$, which immediately leads to the algorithm \textsc{SyzCancel}.

\begin{algorithm}
	\KwIn{A left fraction $(s,r)\in S^{-1}R$.}
	\KwOut{A set of representations of $(s,r)$ containing $C_{s,r}$.}
	\Begin{
		%compute $M:=\operatorname{RSyz}\left(\begin{bmatrix}s&r\end{bmatrix}\right)=\left\{\begin{bmatrix}a&b\end{bmatrix}^T\in R^{2\times1}\mid sa+rb=0\right\}$, the right syzygy module of $\begin{bmatrix}s&r\end{bmatrix}\in R^{1\times2}$\;
		compute $M:=\operatorname{RSyz}(\begin{bmatrix}s&r\end{bmatrix})=\{\begin{bmatrix}a&b\end{bmatrix}^T\in R^{2\times1}\mid sa+rb=0\}$\; %, the right syzygy module of $\begin{bmatrix}s&r\end{bmatrix}\in R^{1\times2}$\;
		choose any non-zero $\begin{bmatrix}a&b\end{bmatrix}^T\in M$\;
		%compute $N:=\operatorname{LSyz}\left(\begin{bmatrix}a&b\end{bmatrix}^T\right)=\left\{\begin{bmatrix}q&p\end{bmatrix}\in R^{1\times2}\mid qa+pb=0\right\}$, the left syzygy module of $\begin{bmatrix}a&b\end{bmatrix}^T$\;
		compute $N:=\operatorname{LSyz}(\begin{bmatrix}a&b\end{bmatrix}^T)=\{\begin{bmatrix}q&p\end{bmatrix}\in R^{1\times2}\mid qa+pb=0\}$\; %, the left syzygy module of $\begin{bmatrix}a&b\end{bmatrix}^T\in R^{2\times1}$\;
		compute $\tilde{N}:=N\cap(S\times R)$\;
		\Return{$\tilde{N}$}\;
	}
	\caption{\textsc{SyzCancel}}
\end{algorithm}

\subsubsection{Factorization-based canceling}

Since $G$-algebras are finite factorization domains (\cite{BHL-FFD}) there are only finitely many factorizations of the denominator $s$.
Thus we can compute $C_{s,r}$ as follows:

\begin{enumerate}[1.]
	\item
		Set $M:=\emptyset$. %$M:=\{(s,r)\}$.
	\item
		Compute all factorizations of $s\in S \subsetneq R$ of the form $s=f_i s_i$, where $f_i$ is irreducible, $s_i$ a non-unit and $i\in I$, where $I$ is a suitable index set for keeping track of these different factorizations.
		If $I=\emptyset$ return $\{(s,r)\}$.
	\item
		Compute the index set $J:=\{i\in I\mid\exists r_i\in R:r=f_ir_i\}$, where $r_i$ can be obtained by right division: $j\in J$ if and only if the right normal form $\operatorname{rightNF}(r,f_j)=0$.
		This can only be the case if $\operatorname{tdeg}(r)\geq\operatorname{tdeg}(f_j)$.
		If $J=\emptyset$ return $\{(s,r)\}$.
	\item
		For every $j\in J$ apply the procedure recursively to $(s_j,r_j)$ and add the results to $M$.
\end{enumerate}

After finitely many steps we obtain a list of all fully canceled representations of $(s,r)$.
Apart from $(s,r)$ itself they all have denominators with total degree strictly smaller than $\operatorname{tdeg}(s)$, since $\operatorname{tdeg}(s_i)=\operatorname{tdeg}(s)-\operatorname{tdeg}(f_i)\leq\operatorname{tdeg}(s) - 1$.
Still there can be several representatives with minimal total degree of the denominator:

\begin{example}
	Consider again the localization $\operatorname{LSat}(V)^{-1}A_1$ from \Cref{case_study_Weyl}.
	Then $(x^2,x\partial-1)$ and $(x\partial+2,\partial^2)$ represent the same fraction in $\operatorname{LSat}(V)^{-1}A_1$ since
	\[
		(x^2,x\partial-1)
		=(\partial x^2,\partial(x\partial-1))
		=(x(x\partial+2),x\partial^2)
		=(x\partial+2,\partial^2).
	\]
	Both denominators have total degree $2$ and cannot be canceled further.
\end{example}

\section{Implementation}

In this section we outline the structure of \texttt{olga.lib}\footnote{The latest version can be found at \url{www.math.rwth-aachen.de/~Johannes.Hoffmann/singular.html} and will also be included in a later version of \textsc{Singular}.} (short for ``Ore localization in $G$-algebras''), our implementation of the algorithms developed above in the computer algebra system \textsc{Singular:Plural}.

\subsection{Setting, conventions and restrictions}

For now \texttt{olga.lib} can perform computations in the following situations:

\begin{itemize}
	\item
		For monoidal localizations, consider a $G$-algebra $A$ generated by the variables $x_1,\ldots,x_n$ and let $1\leq k\leq n$ such that $R:=K[x_1,\ldots,x_k]$ is a commutative polynomial subring of $A$.
		Further, let $g_1,\ldots,g_t\in R\setminus\{0\}$ such that $S:=[g_1,\ldots,g_t]$ is a left Ore set in $A$.
	\item
		Geometric localizations are only implemented for Weyl-like algebras $A$, consisting of $2n$ variables, where the first $n$ variables $x_1,\ldots,x_n$ generate a commutative polynomial subring $R:=K[x_1,\ldots,x_n]$ of $A$: let $\mathfrak{p}$ be a prime ideal in $R\subsetneq A$ and set $S:=R\setminus\mathfrak{p}$.
	\item
		For rational localizations, consider a $G$-algebra $A$ generated by the variables $x_1,\ldots,x_n$ and let $1\leq i_1<\ldots< i_k\leq n$ such that $x_{i_1},\ldots,x_{i_k}$ generate a sub-$G$-algebra $B$ of $A$ and $S:=B\setminus\{0\}$ is a left Ore set in $A$.
\end{itemize}
In any of these cases we can perform basic arithmetic in the localization at $S$ constructively.

\begin{remark}
	In the monoidal case, the restriction for $g_1,\ldots,g_t$ to be contained in a commutative polynomial ring is due to the existence of a unique factorization into irreducible elements there, which easily allows to check whether a given element $s$ is contained in $S$ or not.
	For the computation of the left Ore condition it suffices if the $g_i$ commute pairwise, see Lemma \ref{localization_at_product_is_isomorphic_to_localization_at_factors}.
	
	All computations will actually be carried out in the localization at $[g]$, where $g$ is the square-free part of $g_1\cdot\ldots\cdot g_t$, which is isomorphic to the localization at $S=[g_1,\ldots,g_t]$ again by Lemma \ref{localization_at_product_is_isomorphic_to_localization_at_factors}.
\end{remark}

\begin{remark}
	In the rational case we also need the existence of an elimination ordering for the variables not indexed by $i_1,\ldots,i_k$ to compute the intersection of $\ker(\varphi_{s,r})$ with the subalgebra $B$.
	This technical condition is satisfied in many applications, especially in the transition from polynomial to rational coefficients in the setting of linear functional operators.
	Total rational localizations, that is, computations in the quotient field of a $G$-algebra, also satisfy this condition since no elimination is required.
\end{remark}

\begin{remark}
	While we know that geometric localization is well-defined at any prime ideal $\mathfrak{p}$ in the Weyl setting, in the other situations we have no automatic way to check if the given input indeed represents a left Ore set in $A$.
	If it is not left Ore, then the behaviour of the algorithms is unspecified: computations may yield a plausible result or fail with an error.
\end{remark}

\begin{remark}
	Computing a left representation from a right representation is based on the computation of a left Ore condition.
	Analogously, computing a right representation from a left representation requires the right Ore property.
	If the respective Ore property does not hold, the corresponding algorithms might fail.
\end{remark}

\subsection{Data structure}

A non-commutative fraction $x$ is represented by a \texttt{vector} $[s,r,p,t]$ with entries of type \texttt{poly}, where $(s,r)=s^{-1}r$ is a representation of $x$ as a left fraction, while $(p,t)=pt^{-1}$ is a representation of $x$ as a right fraction.
If $s=0$ or $t=0$ then the corresponding representation is considered as not yet known. If both are zero, the fraction is not valid.
If both $s$ and $t$ are non-zero, the two representations have to agree, that is, $rt=sp$.
A \texttt{vector} adhering to this specifications will be called a \emph{fraction vector} below.

To interpret $x$ in the context of a localization we have to specify a localization type, which is an \texttt{int} with value $0$ for monoidal, $1$ for geometric and $2$ for rational localization, as well as some localization data which is stored in an object of the universal type \texttt{def} to accommodate the different settings:

\begin{itemize}
	\item
		Monoidal: list $g_1,\ldots,g_t$ with entries of type \texttt{poly}.
	\item
		Geometric: \texttt{ideal} $\mathfrak{p}$ in $K[x_1,\ldots,x_n]$.
	\item
		Rational: \texttt{intvec} containing $i_1,\ldots,i_k$.
\end{itemize}

\subsection{Procedures}

Apart from some auxiliary functions, \texttt{olga.lib} contains the following procedures, which require two parameters specifying a left Ore $S$ set via an \texttt{int locType} and a \texttt{def} \texttt{locData} as described in the section above.
If they are not mentioned explicitly they have to be appended at the end of the parameter lists.

\subsubsection{\texttt{ore(poly s, poly r, int locType, def locData, int rightOre)}}

If \texttt{rightOre} is set to $0$, computes $(\tilde{s},\tilde{r},J)$, where $\tilde{s}\in S$ and $\tilde{r}\in R$ satisfy $\tilde{s}r=\tilde{r}s$ and $J$ is an \texttt{ideal} describing all possible choices for $\tilde{s}$ as specified in Section \ref{section_intersection}.
If \texttt{rightOre} is set to $1$, computes the right-sided analogue.
This procedure will be replaced by two separate functions \texttt{leftOre} and \texttt{rightOre} in future releases.

\subsubsection{\texttt{convertRightToLeftFraction(vector v)}}

Computes a right representation of the left fraction \texttt{v}.

\subsubsection{\texttt{convertLeftToRightFraction(vector)}}

\subsubsection{\texttt{addLeftFractions(vector a, vector b)}}

\subsubsection{\texttt{multiplyLeftFractions(vector a, vector b)}}

\subsubsection{\texttt{areEqualLeftFractions(vector a, vector b)}}

\subsubsection{\texttt{isInS(poly p)}}

Checks if \texttt{p} is contained in $S$.

\subsubsection{\texttt{isInvertibleLeftFraction(vector v)}}

Checks if \texttt{v} is invertible (see Section \ref{section_computing_inverse} for the interpretation of the result).

\subsubsection{\texttt{invertLeftFraction(vector v)}}

Returns the inverse of \texttt{v} if \texttt{v} is invertible according to \texttt{isInvertibleLeftFraction}.

\subsubsection{\texttt{cancelLeftFraction(vector v)}}

Performs steps to find an ``easier'' representation of \texttt{v}.

\subsubsection{\texttt{reduceLeftFraction(vector a, vector b)}}

Only for rational localizations: performs a Gr\"obner-like reduction step to reduce \texttt{a} with respect to \texttt{b}.

\subsection{Examples}

The first example demonstrates a left-to-right conversion in the second rational $q$-shift algebra $A$, which is $\mathbb{Q}(q)(x,y)\langle Q_x,Q_y\mid F\rangle$ with the set of relations (cf. also \Cref{exAlgebrasOfOps})
\[
	F=\{ Q_x g(x,y) = g(qx,y) Q_x, Q_y g(x,y) = g(x,qy) Q_y, Q_y Q_x = Q_x Q_y \;
	\mid g(x,y)\in\mathbb{Q}(q)(x,y)\setminus \mathbb{Q}(q) 
	\}.
\]

\begin{verbatim}
	LIB "olga.lib";
	ring Q = (0,q),(x,y,Qx,Qy),dp; // comm. polynomial ring
	matrix C[4][4] = UpOneMatrix(4); // sets non-comm.
	C[1,3] = q; C[2,4] = q;          // relations
	def ncQ = nc_algebra(C,0); // creates A from Q
	setring ncQ;
	intvec v = 1,2;
	poly f = Qx+Qy; poly g = x^2+1;
	vector frac = [g,f,0,0];
	vector result = convertLeftToRightFraction(frac,2,v);
\end{verbatim}

Now \texttt{result} contains the left representation $(x^2+1)^{-1}(Q_x+Q_y)$ of \texttt{frac} as well as its newly computed right representation $(q^4x^2Q_x+x^2Q_y+q^2Q_y)\cdot(x^4+(q^2+1)x^2+q^2)^{-1}$.
We can convince ourselves that the two representations are equal and that the right denominator of \texttt{result} is contained in $S$:

\begin{verbatim}
	f * result[4] == g * result[3];
	-> 1
	isInS(result[4],2,v);
	-> 1
\end{verbatim}

The second example consists of the addition of two left fractions in various localizations of the second Weyl algebra $A_2 = \mathbb{Q}\langle x,y,\partial_x,\partial_y \mid F \rangle$, where the set of relations $F$ is as in \Cref{exAlgebrasOfOps} :

\begin{verbatim}
	LIB "olga.lib";
	ring R = 0,(x,y,dx,dy),dp; // comm. polynomial ring
	def W = Weyl();  setring W; // creates A_2 from R
	poly g1 = x+3; poly g2 = x*y+y;
	list L = g1,g2;
	frac1 = [g1,dx,0,0]; frac2 = [g2,dy,0,0];
	vector resm = addLeftFractions(frac1,frac2,0,L);
\end{verbatim}

Here, \texttt{resm} has left denominator $x^3y+7x^2y+15xy+9y$ and left numerator $x^2y\partial_x+4xy\partial_x+x^2\partial_y+3y\partial_x+6x\partial_y+9\partial_y$ as a fraction in the monoidal localization of $A_2$ at
$S=[x+3, xy+y]$.

\begin{verbatim}
	ideal p = y-3;
	vector resg = addLeftFractions(frac1,frac2,1,p);
\end{verbatim}

\texttt{resg} contains $(x^2y+4xy+3y)^{-1}(xy\partial_x+y\partial_x+x\partial_y+3\partial_y)$ and belongs to the geometric localization of $A_2$ at
$S=\mathbb{Q}[x,y]\setminus \langle y-3 \rangle$.

\begin{verbatim}
	intvec rat = 2,4;
	frac1 = [y+3,dx,0,0]; frac2 = [dy-1,x,0,0];
	vector resr = addLeftFractions(frac1,frac2,2,rat);
\end{verbatim}

Lastly, \texttt{resr} has left denominator $y\partial_y^2-2y\partial_y+3\partial_y^2+y-4\partial_y+1$ and left numerator $x^2y\partial_y+\partial_x\partial_y^2-xy+3x\partial_y-2\partial_x\partial_y-x+\partial_x$, living in the rational localization of $A_2$ at $S=\mathbb{Q}\langle y, \partial_y \mid \partial_y y = y \partial_y + 1 \rangle\setminus\{0\}$.

Lastly, \texttt{resr} is given by
\[
	(y\partial_y^2-2y\partial_y+3\partial_y^2+y-4\partial_y+1,x^2y\partial_y+\partial_x\partial_y^2-xy+3x\partial_y-2\partial_x\partial_y-x+\partial_x),
\]
it is an element of the rational localization of $A_2$ at $S=\mathbb{Q}\langle y,\partial_y\mid\partial_yy=y\partial_y+1\rangle\setminus\{0\}$, which can be written as $\operatorname{Quot}(\mathbb{Q}\langle y,\partial_y\mid\partial_yy=y\partial_y+1\rangle)\langle x,\partial_x\mid\partial_xx=x\partial_x+1\rangle$, a polynomial Weyl algebra in variables $\{x, \partial_x\}$ over the quotient field of a Weyl algebra in $\{y,\partial_y\}$.

\section{Conclusion and future work}

The algorithmic framework presented here is based on a constructive approach that strives for broad generality.
At a very general level we face the major problem of intersecting a left ideal with a submonoid $S$ of $R$.
We are not aware whether this problem is decidable in general. 
Nevertheless, we have proposed solutions for three application-inspired situations where $S$ has additional structure, but even there some restrictions still apply.
This should not be considered as a failure of the approach, but rather as a hint at the high intrinsic complexity of the problem.

The proposed framework is easily expandable to include other types of left Ore sets $S \subsetneq R$ provided the following two problems can be solved algorithmically:
\begin{enumerate}
\item the submonoid membership problem (i.\,e. whether $r\in S$ for a given $r\in R$),
%(that is, deciding whether a given element of $R$ is contained in $S$),
\item the intersection of a left ideal in $R$ with a submonoid $S$.
\end{enumerate}

Apart from overcoming the current restrictions already mentioned throughout the text, we are working on the following:

The section about the left saturation closure of multiplicatively closed sets is only a special case of a more general notion which also includes the important concept of local closure of submodules, such 
as the celebrated Weyl closure in $D$-module theory.

Utilizing the ability to create user-defined data types introduced in \textsc{Singular} from version $4$ on, we are working on an object-oriented interface for \texttt{olga.lib} to improve usability.
To this end, we also intend to turn \texttt{olga.lib} into a true sandbox environment for all computations associated with Ore localization of $G$-algebras.

\section{Acknowledgements}

The authors are very grateful to Daniel Andres, Vladimir Bavula, Burcin Erocal, Christoph Koutschan and Oleksander Motsak for discussions on the subject, even if some of these have happened a couple of years ago.
Furthermore, we would like to thank the referees for their helpful suggestions.
The second author is grateful to the transregional collaborative research centre 
SFB-TRR 195 ``Symbolic Tools in Mathematics and their Application'' of the German DFG
for partial financial support.

\bibliography{bibliography}
\bibliographystyle{plain}

\end{document}